\documentclass[
a4paper,
reqno,
11pt,
]{article}
\usepackage{
amssymb,
amsmath,
amsthm,
eucal,
dsfont,
}
\makeatletter
 
  \@addtoreset{equation}{section}
 \makeatother
\theoremstyle{plain}

\newtheorem{theorem}{Theorem}[section]
\newtheorem{lemma}[theorem]{Lemma}
\newtheorem{proposition}[theorem]{Proposition}

\theoremstyle{definition}

\newtheorem{assumption}{Assumption}

\newcommand{\bignorm}[1]{{\Big|\Big|#1\Big|\Big|}}
\newcommand{\norm}[1]{{||#1||}}

\newcommand{\wtilde}[1]{{\widetilde{#1}}}
\def\supp{\mathop{\mathrm{supp}}\nolimits}

\def\loc{\mathop{\mathrm{loc}}\nolimits}

\def\R{{\mathbb{R}}}
\def\Z{{\mathbb{Z}}}

\def\C{{\mathbb{C}}}

\def\S{{\mathcal{S}}}

\def\H{{\mathcal{H}}}
\def\A{{\mathcal{A}}}
\def\B{{\mathcal{B}}}

\def\X{{\mathcal{X}}}

\def\<{{\langle}}
\def\>{{\rangle}}

\def\ep{{\varepsilon}}

\title
{Global-in-time smoothing effects for Schr\"odinger equations with inverse-square potentials}
\author{Haruya Mizutani\footnote{Department of Mathematics, Graduate School of Science, Osaka University, Toyonaka, Osaka 560-0043, Japan. E-mail address: \texttt{haruya@math.sci.osaka-u.ac.jp}}
}

\date{\empty}

\begin{document}
\maketitle

\begin{abstract}
The purpose of this note is to prove global-in-time smoothing effects for the Schr\"odinger equation with potentials exhibiting critical singularity. A typical example of  admissible potentials is the inverse-square potential $a|x|^{-2}$ with $a>-(n-2)^2/4$. This particularly gives an affirmative answer to a question raised by \cite{BDDLL}. The proof employs a uniform resolvent estimate proved by \cite{BVZ} and an abstract perturbation method by \cite{BoMi}. 
\end{abstract}


%
\footnotetext{2010 \textit{Mathematics Subject Classification}. Primary 35Q41; Secondary 35B45.}\footnotetext{\textit{Key words and phrases}. smoothing estimate, Strichartz estimate, Schr\"odinger equation, inverse-square potential}

\section{Introduction}
\label{section_Introduction}
This note is concerned with smoothing properties of the time-dependent Schr\"odinger equation
\begin{align}
\label{equation}
i\partial_t u(t,x)=Hu(t,x)+F(t,x);\quad u(0,x)=\psi(x),
\end{align}
with given data $\psi\in L^2(\R^n)$ and $F\in L^1_{\mathrm{loc}}(\R;L^2(\R^n))$, where $H=-\Delta+V(x)$ is a Schr\"odinger operator  on $\R^n$, $n\ge3$, with a real-valued function $V$ which decays at spatial infinity in a suitable sense and  has a critical singularity at the origin. A typical example of potentials we have in mind is the inverse-square potential $V(x)=a|x|^{-2}$ satisfying  $a>-(n-2)^2/4$. 

Let us first recall several known results for the free case, describing the motivation of this paper. It is well-known that the solution $u=e^{it\Delta}\psi$ to the free Schr\"odinger equation
$$
i\partial_t u(t,x)=-\Delta u(t,x);\quad u|_{t=0}=\psi\in L^2(\R^n),
$$
satisfies the following global-in-time smoothing effect
\begin{align}
\label{intro_1}
\norm{\<x\>^{-\rho}|D|^{1/2}e^{it\Delta}\psi}_{L^2(\R^{1+n})}\le C\norm{\psi}_{L^2(\R^n)}, 
\end{align}
where $\<x\>=(1+|x|^2)^{1/2}$, $\rho>1/2$ and $|D|=(-\Delta)^{1/2}$ (see Ben-Artzi and Klainerman \cite{BeKl} for $n\ge3$ and Chihara \cite{Chi} for $n=2$). When $n\ge3$, the estimate of the form
\begin{align}
\label{intro_2}
\norm{w(x)e^{it\Delta}\psi}_{L^2(\R^{1+n})}\le C\norm{w}_{L^{n}(\R^n)}\norm{\psi}_{L^2(\R^n)}
\end{align}
was proved by Kato and Yajima \cite{KaYa}. The estimate \eqref{intro_2} also follows from H\"older's inequality and the endpoint Strichartz estimate proved by Keel and Tao \cite{KeTa}: 
\begin{align}
\label{intro_3}
\norm{e^{it\Delta}\psi}_{L^2(\R;L^{\frac{2n}{n-2}}(\R^n))}\le C\norm{\psi}_{L^2(\R^n)},
\end{align}
which can be also regarded as a smoothing property in $L^p$-spaces. All of these three estimates are fundamental tools in the study of Cauchy problem and scattering theory for both linear and nonlinear Schr\"odinger equations (see \cite{Kat,KaYa,Caz,Tao,RuSu1} and references therein). 
It is also worth noting that \eqref{intro_1} and \eqref{intro_2} are closely connected with uniform estimates for the resolvent $(-\Delta-z)^{-1}$ with respect to $z\in \C\setminus [0,\infty)$  
(see the next section for more details). 

There is a vast literature on extending these estimates \eqref{intro_1}--\eqref{intro_3} to the Schr\"odinger operator $H=-\Delta+V$ with potential $V(x)$.  For the case when $V$ has enough regularity and decays sufficiently fast at spatial infinity, we refer to \cite{RoSc,FaVe, MMT,Moc,Gol} and references therein. There are also several results in the case when $V$ has critical singularity. In particular, the Schr\"odinger operator with the inverse-square potential of the form\
$$
H_a=-\Delta+a|x|^{-2},\quad a>-\frac{(n-2)^2}{4},
$$
has attracted increasing attention since it represents a borderline case  for the validity of \eqref{intro_1}--\eqref{intro_3} (\cite{Duy,GVV}), where note that $(n-2)^2/4$ is the best constant in Hardy's inequality
\begin{align}
\label{Hardy}
\frac{(n-2)^2}{4}\int |x|^{-2}|u|^2dx\le \int|\nabla u|^2dx,\quad u\in C_0^\infty(\R^n). 
\end{align}
We refer to \cite{BPST1,BPST2,BVZ,BoMi} for Kato--Yajima type estimates \eqref{intro_2} and to \cite{BPST1,BPST2,BoMi,Miz1} for Strichartz estimates \eqref{intro_3}. Concerning the estimate \eqref{intro_1}, in a recent paper \cite{BDDLL}, the authors showed, among the others, the following
\begin{theorem}	[{\cite[Theorem 1.2]{BDDLL}}]
\label{theorem_BDDLL}
Let $n\ge3$, $a\ge -(n-2)^2/4+1/4$ and $\ep>0$. Then there exists $C_\ep>0$ such that for all $\psi\in L^2(\R^n)$, $e^{-itH_a}\psi$ satisfies
\begin{align}
\label{theorem_BDDLL_1}
\norm{w(|x|)|D|^{1/2}e^{-itH_a}\psi}_{L^2(\R^{1+n})}\le C_\ep\norm{\psi}_{L^2(\R^n)}
\end{align}
where $w(r)=r^{(\ep-1)/2}(1+r^\ep)^{-1}$. 
\end{theorem}
The condition $a\ge -(n-2)^2/4+1/4$ was used to ensure that $w$ satisfies some conditions for two-sided weighted norm estimates $$C_1\norm{w(x)(-\Delta)^{s/2}f}_{L^2}\le \norm{w(x)H_a^{s/2}f}_{L^2}\le C_2 \norm{w(x)(-\Delta)^{s/2}f}_{L^2}$$ established by the same paper (see \cite[Theorem 1.1]{BDDLL}). Then authors raised a question whether  \eqref{theorem_BDDLL_1} holds under the condition $a>-(n-2)^2/4$. The main purpose of the present short note is to give an affirmative answer to this question. More precisely, we prove global-in-time smoothing effect of the form \eqref{theorem_BDDLL_1} for Schr\"odinger operators $H=-\Delta+V(x)$ with a large class of real-valued potentials which particularly includes the inverse-square potential with $a>-(n-2)^2/4$. Furthermore, global-in-time smoothing effects for the solution to \eqref{equation} with the inhomogeneous term $F$ are also studied for the same class of potentials. The proofs are based on an abstract perturbation method by our previous work \cite{BoMi} and a uniform estimate proved by \cite{BVZ} for the weighted resolvent $|x|^{-1}(H-z)^{-1}|x|^{-1}$ with respect to $z\in \C\setminus[0,\infty)$. 

In order to state the main results, we introduce some notation. From now on we let $n\ge3$ and impose the following condition:

\begin{assumption}
\label{assumption_A}
$V(x)$ a real-valued function on $\R^n$ such that $|x| V \in L^{n,\infty}(\R^n)$ and $x \cdot \nabla V \in L^{n/2,\infty}(\R^n)$. Moreover, there exists $\delta>0$ such that $-\Delta+V\ge-\delta\Delta$ and $-\Delta-V-x\cdot \nabla V\ge-\delta \Delta$ in the sense of forms, that is, for all $u\in C_0^\infty(\R^n)$, 
\begin{align}
\label{assumption_A_1}
\<(-\Delta+V)u,u\>\ge \delta\norm{\nabla u}_{L^2},\quad
\<(-\Delta-V-x\cdot (\nabla V))u,u\>\ge \delta\norm{\nabla u}_{L^2}. 
\end{align}
\end{assumption}
Here $\<f,g\>=\int f\overline g dx$ is the inner product in $L^2(\R^n)$ and $L^{p,q}$ is the Lorentz space (see the end of this section). 
A typical example satisfying Assumption \ref{assumption_A} is the inverse-square potential $V(x)=a|x|^{-2}$ with $a>-(n-2)^2/4$. In this case, it follows from Hardy's inequality \eqref{Hardy} that \eqref{assumption_A_1} is satisfied with $\delta=1-4|a|/(n-2)^2>0$ if $a<0$ or $\delta=1$ if $a\ge0$. Moreover, Assumption \ref{assumption_A} is general enough to include some potentials such that $|x|^2V\notin L^\infty$. For instance, we let $c_1,c_2>0$, $\alpha\in\R^n$ and $\chi\in C^1(\R)$ such that $0\le \chi\le 1$ and $|\chi^{(k)}(t)|\le |t|^{-k-1}$ for $|t|\ge1$. Then 
$$
V(x)=\frac{-(n-2)^2/4+c_1}{|x|^{2}}-\frac{c_2\chi(|x-\alpha|)}{|x-\alpha|}
$$
satisfies Assumption \ref{assumption_A} with $\delta=c_1-c_2(2+\sup|\chi'|)(|\alpha|+1)$ if 
$$
0<c_2<\frac{c_1}{(2+\sup|\chi'|)(|\alpha|+1)}. 
$$

Under Assumption \ref{assumption_A}, Hardy's inequality \eqref{Hardy} implies that the sesquilinear form $$Q_H(u,v)=\<(-\Delta+V)u,v\>
,\quad u,v\in C_0^\infty(\R^n),$$ is symmetric, non-negative and closable such that the domain of its closure $\overline{Q}_H$ satisfies $D(\overline{Q}_H)=\mathcal H^1(\R^n)$. 
Let  $H$ be the Friedrichs extension of $\overline Q_H$, $e^{-itH}$ the unitary group on $L^2(\R^n)$ generated by $H$ via Stone's theorem and $\Gamma_H$ the inhomogeneous propagator defined by
\begin{align}
\label{inhomogeneous_propagator}
\Gamma_HF(t)=\int_0^t e^{-i(t-s)H}F(s)ds,\quad F\in L^1_{\mathrm{loc}}(\R;L^2(\R^n)). 
\end{align}
Then a unique (mild) solution to the Schr\"odinger equation \eqref{equation} is given by
\begin{align}
\label{solution}
u(t)=e^{-itH}\psi-i\Gamma_HF(t). 
\end{align}
We say that $v(x)$ belongs to the Muckenhoupt $A_2$ class if $v,1/v\in L^1_{\mathrm{loc}}(\R^n)$, $v\ge0$ and
$$
\Big(\frac{1}{|B|}\int_B v(x)dx\Big)\Big(\frac{1}{|B|}\int_{B}\frac{1}{v(x)}dx\Big)\le C
$$
for all ball $B\subset \R^n$ with some constant $C>0$ independent of $B$. 

The main result in this paper then is as follows.

\begin{theorem}	\label{theorem_1}
Let $n\ge3$, $V$ satisfy Assumption \ref{assumption_A} and $w\in L^2(\R)$. Suppose $w(|x|)^2\in A_2$ and, for any $j=1,2,...,n$, there exists $C_j>0$ such that
$$
w(|x|)\le C_jw(x_j),\quad x=(x_1,...,x_n)\in \R^n. 
$$
Then there exists $C>0$, independent of $w$, such that $e^{-itH}$ satisfies
$$
\norm{w(|x|)|D|^{1/2}e^{-itH}\psi}_{L^{2}(\R^{1+n})}\le C\norm{w}_{L^2(\R)}\norm{\psi}_{L^2(\R^n)},\quad \psi\in L^2(\R^n). 
$$
\end{theorem}

Assuming $0<\ep<1$ without loss of generality, it is easy to see that $w(r)=r^{(\ep-1)/2}(1+r^\ep)^{-1}$ fulfills the above conditions. Another typical example of $w$ is $\<x\>^{-\rho}$ in which case we have

\begin{theorem}	
\label{theorem_2}
Let $n\ge3$, $\rho>1/2$ and $\A,\B\in \{\dot\H^{-1/2,\rho}(\R^n),\ L^{\frac{2n}{n+2},2}(\R^n)\}$.  Suppose $V$ satisfies Assumption \ref{assumption_A}. Then the solution $u$ to \eqref{equation} given by \eqref{solution} satisfies
\begin{align}
\label{theorem_2_1}
\norm{u}_{L^{2}(\R;\B^*)}\le C\norm{\psi}_{L^2(\R^n)}+C\norm{F}_{L^2(\R;\A)}
\end{align}
for all $\psi\in L^2(\R^n)$ and $F\in L^1_{\mathrm{loc}}(\R;L^2(\R^n)) \cap L^2(\R;\A)$. 
\end{theorem}

Here 
$\dot\H^{s,\mu}(\R^n)$ denotes the weighted homogeneous Sobolev space equipped with the norm $\norm{f}_{\dot\H^{s,\mu}}=\norm{\<x\>^{\mu}|D|^{s}f}_{L^2}$. 
Note that $(L^{\frac{2n}{n+2},2}(\R^n))^*=L^{\frac{2n}{n-2},2}(\R^n)$ and $(\dot\H^{-1/2,\rho}(\R^n))^*=\dot\H^{1/2,-\rho}(\R^n)$. If $\A=\B=L^{\frac{2n}{n+2},2}(\R^n)$, \eqref{theorem_2_1} becomes the endpoint Strichartz estimate and was proved by our previous work \cite{BoMi}. If $\A=\B=\dot\H^{-1/2,\rho}(\R^n)$, \eqref{theorem_2_1} is a generalization of \eqref{intro_1} and seems to be new under Assumption \ref{assumption_A}. Here we stress that $\A$ and $\B$ do not have to coincide. 
\\\\
\noindent{\it Notation}. Throughout the paper we use the following notation. 
For $T>0$ and a Banach space $\X$, we denote $\norm{F}_{L^p_T\X}=\norm{F}_{L^p([-T,T];\X)}$. $L^{p,q}(\R^n)$ denotes the Lorentz space equipped with the norm $\norm{\cdot}_{L^{p,q}(\R^n)}$ satisfying $\norm{f}_{L^{p,q}(\R^n)}\sim\norm{t d_f(t)^{1/p}}_{L^q(\R_+,t^{-1}dt)}$, 
where $d_f(t):=\mu(\{x\in \R^n\ |\ |f(x)|>t\})$ is the distribution function of $f$. 
We use the convention $L^{\infty,\infty}=L^\infty$. For $1\le p,p_1,p_2<\infty$ and $1\le q,q_1,q_2\le\infty$ satisfying $1/p=1/p_1+1/p_2$, $1/q=1/q_1+1/q_2$, we have H\"older's inequality for Lorentz spaces:
\begin{equation}
\begin{aligned}
\label{Holder}
\norm{fg}_{L^{p,q}}\le C\norm{f}_{L^{p_1,q_1}}\norm{g}_{L^{p_2,q_2}},\quad
\norm{fg}_{L^{p,q}}\le C\norm{f}_{L^{\infty}}\norm{g}_{L^{p,q}}.
\end{aligned}
\end{equation}
When $n\ge3$, we also have Sobolev's inequality in Lorentz spaces:
\begin{align}
\label{Sobolev}
\norm{f}_{L^{\frac{2n}{n-2},2}}\le C \norm{\nabla f}_{L^2},\quad f\in \H^1. 
\end{align}
We refer to  \cite{Gra1} for more details on Lorentz spaces. In what follows we often omit $\R^n$ from $L^p(\R^n)$ and so on, if there is no confusion. \\

The rest of the paper is organized as follows. We first recall in the next section the abstract perturbation method developed in \cite{BoMi}, which plays an important role in the proof of the main theorems. The proof of Theorems \ref{theorem_1} and \ref{theorem_2} is given in Section \ref{section_proof}. \\\\
\noindent{\bf Acknowledgments.} The author would like to express his sincere gratitude to Jean-Marc Bouclet for valuable discussions and for hospitality at the Institut de Math\'ermatiques de Toulouse, Universit\'e Paul Sabatier, where a part of this work has been done. The author also thank the referee for useful suggestions that helped to improve the presentation of this paper. He is partially supported by JSPS Grant-in-Aid for Young Scientists (B) JP25800083 and by Osaka University Research Abroad Program 150S007. 

\section{An abstract perturbation method}
Here we recall the abstract method developed in \cite{BoMi}. We begin with recalling the notion of the (super)smoothness in the sense of Kato \cite{Kat} and Kato-Yajima \cite{KaYa}. Let $\H$ be a Hilbert space with inner product $\<\cdot,\cdot\>$ and norm $\norm{\cdot}$, $H$ a self-adjoint operator on $\H$ and $A$ a densely defined closed operator on $\H$. Note that $A^*$ is also a densely defined closed operator (see \cite[Theorem VIII.1]{ReSi}). Let $R_H(z):=(H-z)^{-1}$, $z\notin\sigma(H)$. Then we say that $A$ is {\it $H$-smooth} with bound $a$ if 
$$
\sup_{z\in \C\setminus\R}|\<(R_H(z)-R_H(\overline z))A^*\psi,A^*\psi\>|\le \frac{a^2}{2}\norm{\psi}^2,\quad \psi\in D(A^*). 
$$
We say that $A$ is {\it $H$-supersmooth} with bound $a$ if 
$$
\sup_{z\in \C\setminus\R}|\<R_H(z)A^*\psi,A^*\psi\>|\le \frac{a}{2}\norm{\psi}^2,\quad \psi\in D(A^*). 
$$
Note that if $A$ is $H$-supersmooth with bound $a$ then $A$ is $H$-smooth with bound $(2a)^{1/2}$. 
The $H$-(super)smoothness is closely connected with smoothing effects.

\begin{proposition}
\label{proposition_abstract_1}
{\rm (1)} $A$ is $H$-smooth with bound $a$ if and only if,  for any $\psi\in \H$, $e^{-itH}\psi$ belongs to $D(A)$ for a.e. $t\in \R$ and 
\begin{align}
\label{proposition_abstract_1_1}
\norm{Ae^{-itH}\psi}_{L^2(\R;\H)}
\le a\norm{\psi}. 
\end{align}
{\rm (2)} Suppose $A$ is $H$-supersmooth with bound $a$. Then, for any simple function $F:\R\to D(A^*)$ and $t\in \R$, $Ae^{-i(t-s)H}A^*F(s)$ is Bochner integrable in $s$ over $[0,t]$ (or $[t,0]$) and satisfies 
\begin{align}
\label{proposition_abstract_1_2}
\bignorm{e^{-|\ep t|}\int_0^t Ae^{-i(t-s)H}A^*F(s)ds}_{L^2(\R;\H)}\le a\norm{e^{-|\ep t|}F}_{L^2(\R;\H)}
\end{align}
for all $\ep\in \R$. Conversely, if the estimate \eqref{proposition_abstract_1_2}  holds for all simple function $F:\R\to D(A^*)$ and $|\ep|<\ep_0$ with some $\ep_0>0$, then $A$ is $H$-supersmooth with bound $a$. 
\end{proposition}

\begin{proof}
The first statement is due to \cite[Lemma 3.6 and Theorem 5.1]{Kat} (see also \cite[Theorem XIII.25]{ReSi}). The second assertion was proved by \cite[Theorem 2.3]{Dan}. 
\end{proof}

Note that if $A$ is $H$-smooth then $A$ is infinitesimally $H$-bounded (\cite[Theorem XIII.22]{ReSi}). 
Also note that the estimate \eqref{proposition_abstract_1_2} can be replaced by 
\begin{align}
\label{proposition_abstract_1_3}
\bignorm{e^{-|\ep t|}\int_0^t Ae^{-i(t-s)H}A^*F(s)ds}_{L^2_T\H}\le a\norm{e^{-|\ep t|}F}_{L^2_T\H}. 
\end{align}
Indeed, \eqref{proposition_abstract_1_2} implies \eqref{proposition_abstract_1_3} since $s\in [-T,T]$ if $t\in[-T,T]$ and $s\in[0,t]$ (or $s\in [-t,0]$). Conversely, since \eqref{proposition_abstract_1_3} implies
$$
\bignorm{e^{-|\ep t|}\int_0^t Ae^{-i(t-s)H}A^*F(s)ds}_{L^2_T\H}\le a\norm{e^{-|\ep t|}F}_{L^2(\R;\H)}
$$
and $a$ is independent of $T$, one has \eqref{proposition_abstract_1_2} by letting $T\to\infty$. 
Let $\Gamma_H$ be the inhomogeneous propagator defined by the formula \eqref{inhomogeneous_propagator} and set
\begin{align}
\label{adjoint_1}
\Gamma_H^*F(t)=\mathds1_{[0,\infty)}(t)\int _t^Te^{-i(t-s)H}F(s)ds-\mathds1_{(-\infty,0]}(t)\int_{-T}^te^{-i(t-s)H}F(s)ds. 
\end{align}
By a direct calculation, $\<\<\Gamma_HF,G\>\>_T=\<\<F,\Gamma_H^*G\>\>_T$ for $F,G\in L^1_{\loc}\H$, where 
$$
\<\<F,G\>\>_T:=\int_{-T}^T\<F(t),G(t)\>dt. 
$$
Hence $\Gamma_H^*$ is the adjoint of $\Gamma_H$ in $L^2_T\H$. In the abstract theorem below, the operators $A\Gamma_H$ and $A\Gamma_H^*$ for some $H$-smooth operator $A$ play important roles. These operators are a priori well-defined on $L^1_{\loc}(\R;D(H))$ since, for some $z\notin \sigma(H)$ and each $T>0$, 
$$
\norm{A\Gamma_HF}_{L^2_T\H}+\norm{A\Gamma^*_HF}_{L^2_T\H}\le C_T\norm{AR_0(z)}_{\mathbb B(\H)}\norm{R_0(z)F}_{L^1_T\H}<\infty
$$
The next lemma  provides a rigorous definition of $A\Gamma_HF(t)$ for $F\in L^1_{\loc}(\R;\H)$. 

\begin{lemma}[{\cite[Lemma 4.3 and Lemma 4.5]{BoMi}}]
\label{lemma_abstract_2}
Suppose that $A$ is $H$-smooth with bound $a$. Then $A\Gamma_H$ and $A\Gamma^*_H$ extend to bounded operators from $L^1_T\H$ to $L^2_T\H$ such that
\begin{align}
\label{lemma_abstract_2_1}
\norm{A\Gamma_HF}_{L^2_T\H}\le Ca\norm{F}_{L^1_T\H},\quad
\norm{A\Gamma_H^*F}_{L^2_T\H}\le Ca\norm{F}_{L^1_T\H}
\end{align}
for any $T>0$ and $F\in L^1_T\H$ with some $C>0$ independent of $a,T$ and $F$. Moreover, we have
\begin{align}
\label{lemma_abstract_2_2}
A\Gamma_HF(t)=\int_0^tAe^{-i(t-s)H}F(s)ds
\end{align}
for all simple function $F:[-T,T]\to \H$ and a.e. $t\in[-T,T]$. In particular $A\Gamma_HF(t)$ and $\int_0^tAe^{-i(t-s)H}F(s)ds$ coincide in $L^2_T\H$. 
\end{lemma}

\begin{proof}
For the sake of self-containedness, we give the proof for $\Gamma_H$ in detail. The proof for $\Gamma_H^*$ is analogous. Let $\chi\in C_0^\infty(\R)$ be such that $\chi\equiv1$ near $0$ and $0\le\chi\le1$ and set $\chi_\ep(t)=\chi(\ep t)$. Note that $A\Gamma_H\chi_\ep(H)=A\chi_\ep(H)\Gamma_H$ is well-defined on $L^1_T\H$ since $A$ is $H$-bounded. Moreover, the $H$-smoothness of $A$ shows$$\bignorm{Ae^{-itH}\int_{[0,T]}e^{isH}\chi_\ep(H)F(s)ds}_{L^2_T\H}\le a\norm{F}_{L^1_T\H}.$$Using the Christ-Kiselev lemma \cite{ChKi}, one can replace $[0,T]$ by $[0,t]$ in the left hand side to obtain$$\norm{A\chi_\ep(H)\Gamma_HF}_{L^2_T\H}\le Ca\norm{F}_{L^1_T\H}$$with some universal constant $C>0$ independent of $F$. This estimate implies that, for all $F\in L^1_T\H$, $A\chi_\ep(H)\Gamma_HF$ converges in $L^2_T\H$ as $\ep\to0$ and the limit denoted by the same symbol $A\Gamma_HF$ satisfies the first estimate in \eqref{lemma_abstract_2_1}. This proves the half part of this lemma. 

Next, let $F:[-T,T]\to \H$ be a simple function. As we have shown, $A\chi_{\ep_n}(H)\Gamma_H$ converges to $A\Gamma_HF$ in $L^2_T\H$ as $n\to\infty$ for any sequence $\ep_n>0$ with $\ep_n\to0$. Then one can find a subsequence $\ep_{n_k}$ and a null set $\mathcal N\subset [-T,T]$ such that $A\chi_{\ep_{n_k}}(H)\Gamma_HF\to A\Gamma_HF(t)$ in $\H$ as $k\to\infty$ for all $t\in [-T,T]\setminus \mathcal N$. Hence, in order to show \eqref{lemma_abstract_2_2}, it suffices to check that
$$
\bignorm{\int_0^tAe^{-i(t-s)H}F(s)ds-A\chi_{\ep_{n_k}}(H)\Gamma_HF(t)}_\H\to0
$$
as $k\to\infty$ for all $t\in [-T,T]\setminus \mathcal N$. Let us write $F=\sum_{j=1}^N\mathds1_{E_j}(t)f_j$ with $f_j\in \H$ and measurable sets $E_j\subset [-T,T]$. Then we have for all $t\in [-T,T]\setminus \mathcal N$,
\begin{align*}
&\bignorm{\int_0^tAe^{-i(t-s)H}F(s)ds-A\chi_{\ep_{n_k}}(H)\Gamma_HF(t)}_\H\\
&\le \sum_{j=1}^N\bignorm{\int_{[0,t]\cap M_j}Ae^{isH}(1-\chi_{\ep_{n_k}}(H))e^{-itH}f_jds}_\H\\
&\le \sum_{j=1}^NT^{1/2}\Big(\int_\R\norm{Ae^{isH}e^{-itH}(1-\chi_{\ep_{n_k}}(H))f_j}_{\H}^2ds\Big)^{1/2}\\
&\le CT^{1/2}\sum_{j=1}^N\norm{(1-\chi_{\ep_{n_k}}(H))f_j}_{\H}\to 0
\end{align*}
as $k\to\infty$, where we have used H\"older's inequality in the third line and the $H$-smoothness of $A$ and the unitarity of $e^{-itH}$ in the last line. This completes the proof. 
\end{proof}

In what follows, $A\Gamma_H$ and $A\Gamma_H^*$ denote such extensions. By Proposition \ref{proposition_abstract_1} (1), \eqref{proposition_abstract_1_3} and this lemma, 
if $A$ is $H$-supersmooth with bound $a$ then $A\Gamma_HA^*$ is well-defined on $L^1_{\loc}(\R;(D(A^*))$ and satisfies
$
\norm{A\Gamma_HA^*F}_{L^2_T\H}\le a\norm{F}_{L^2_T\H}
$
for all simple function $F:[-T,T]\to D(A^*)$. 

Now we are ready to recall an abstract theorem by \cite{BoMi}. Let $(H_0,H)$ be a pair of self-adjoint operators on $\H$ such that $H=H_0+V_1^*V_2$ in the following sense:
\begin{itemize}
\item $V_1,V_2$ are densely defined closed operators on $\H$ such that we have continuous embeddings $D(H_0)\subset D(V_1)$, $D(H)\subset D(V_1)$ and $D(H)\subset D(V_2)$. 
\item $\<Hf,g\>=\<f,H_0g\>+\<V_2f,V_1g\>$ for $f\in D(H)$ and $g\in D(H_0)$. 
\end{itemize}
Note that, under the above conditions, $V_1^*,V_2^*$ are also densely defined. Recall that a couple of two Banach spaces $(\A,\B)$ is said to be a Banach couple if both $\A$ and $\B$ are algebraically and topologically embedded in a Hausdorff topological vector space $\mathcal C$. 
\begin{proposition}[{\cite[Theorem 4.7]{BoMi}}]
\label{proposition_abstract_2}
Let $\A,\B$ be two Banach spaces such that $(\A,\H)$ and $(\B,\H)$ are Banach couples.  
Suppose that $V_1$ is $H_0$-smooth and $V_2$ is both $H_0$-smooth and $H$-smooth. Consider the following series of estimates:  
\begin{align}
\label{proposition_abstract_2_1}
|\<\<e^{-itH_0}\psi,G\>\>_T|&\le s_1\norm{\psi}_{\H}\norm{G}_{L^2_T\B}, \\
\label{proposition_abstract_2_2}
|\<\<\Gamma_{H_0}F,G\>\>_T|&\le s_2\norm{F}_{L^2_T\A}\norm{G}_{L^2_T\B},\\\label{proposition_abstract_2_3}
\norm{V_1\Gamma_{H_0}^*G}_{L^2_T\H}&\le s_3\norm{G}_{L^2_T\B},\\
\label{proposition_abstract_2_4}
\norm{V_1\Gamma_{H_0}F}_{L^2_T\H}&\le s_4\norm{F}_{L^2_T\A},\\
\label{proposition_abstract_2_5}
\norm{V_2\Gamma_{H_0}F}_{L^2_T\H}&\le s_5\norm{F}_{L^2_T\A},\\
\label{proposition_abstract_2_6}
\norm{V_2e^{-itH}\psi}_{L^2_T\H}&\le s_6\norm{\psi}_{\H},\\
\label{proposition_abstract_2_7}
\norm{V_2\Gamma_HV_2^*\wtilde G}_{L^2_T\H}&\le s_7\norm{\wtilde G}_{L^2_T\H}. 
\end{align}
{\rm (1)} Suppose there exist constants $s_1,s_3,s_6>0$ such that \eqref{proposition_abstract_2_1}, \eqref{proposition_abstract_2_3} and \eqref{proposition_abstract_2_6} are satisfied for all $\psi\in \H$ and simple function $G:[-T,T]\to \H\cap \B$:
Then one has
$$
|\<\<e^{-itH}\psi,G\>\>_T|\le (s_1+s_3s_6)\norm{\psi}_{\H}\norm{G}_{L^2_T\B}
$$
for all $\psi\in \H$ and simple function $G:[-T,T]\to \H\cap \B$. \\
{\rm (2)} Suppose that there exist constants $s_2,s_3,s_4,s_5,s_7>0$ such that \eqref{proposition_abstract_2_2}, \eqref{proposition_abstract_2_3}, \eqref{proposition_abstract_2_4}, \eqref{proposition_abstract_2_5} and \eqref{proposition_abstract_2_7} hold for all simple functions $F:[-T,T]\to \H\cap \A$, $G:[-T,T]\to \H\cap \B$ and $\wtilde G:[-T,T]\to D(H)$. Then one has
$$
|\<\<\Gamma_{H}F,G\>\>_T|\le (s_2+s_3s_5+s_3s_4s_7)\norm{F}_{L^2_T\A}\norm{G}_{L^2_T\B}
$$
for all simple functions $F:[-T,T]\to \H\cap \A$ and $G:[-T,T]\to \H\cap \B$. 
\end{proposition}

\begin{proof}
We give the proof in detail for the sake of self-containedness. Let us first show the first statement. The Duhamel formula implies
\begin{align}
\label{proof_proposition_abstract_2_1}
\<e^{-itH}\psi,\varphi\>=\<e^{-itH_0}\psi,\varphi\>-i\int_0^t\<V_2e^{-irH}\psi,V_1e^{-i(r-t)H_0}\varphi\>dr,\quad \psi,\varphi\in L^2. 
\end{align}
Plugging $\varphi=G(t)$, integrating over $t\in [-T,T]$ and using Fubini's theorem, we learn by the formula \eqref{adjoint_1} of the adjoint $\Gamma_{H_0}^*$ that
$$
\<\<e^{-itH}\psi,G\>\>_T=\<\<e^{-itH_0}\psi,G\>\>_T-i\<\<V_2e^{-itH}\psi,V_1\Gamma_{H_0}^*G\>\>_T. 
$$
Applying \eqref{proposition_abstract_2_1}, \eqref{proposition_abstract_2_3} and \eqref{proposition_abstract_2_6}, we then obtain the first assertion (1). 

In order to prove the  second assertion (2), we replace $t$ by $t-s$ and plug $\psi=F(s),\varphi=G(t)$ and integrate over $s\in [0,t]$ in \eqref{proof_proposition_abstract_2_1} to obtain
\begin{align*}
\<\Gamma_HF(t),G(t)\>
&=\<\Gamma_{H_0}F(t),G(t)\>-i\int_0^t\int_s^t\<V_2e^{-i(\tau-s)H}F(s),V_1e^{-i(\tau-t)H_0}G(t)\>d\tau ds\\
&=\<\Gamma_{H_0}F(t),G(t)\>-i\int_0^t\<V_2\Gamma_HF(\tau),V_1e^{-i(\tau-t)H_0}G(t)\>d\tau.
\end{align*}
As above, integrating in $t\in [-T,T]$ and using  \eqref{adjoint_1} implies
\begin{align}
\label{proof_proposition_abstract_2_2}
\<\<\Gamma_HF,G\>\>_T=\<\<\Gamma_{H_0}F,G\>\>_T-i\<\<V_2\Gamma_HF,V_1\Gamma_{H_0}^*G\>\>_T. 
\end{align}
Exchanging the roles of $H$ and $H_0$, we also obtain
\begin{align}
\label{proof_proposition_abstract_2_3}
\<\<\Gamma_HF,G\>\>_T=\<\<\Gamma_{H_0}F,G\>\>_T-i\<\<V_1\Gamma_{H_0}F,V_2\Gamma_{H}^*G\>\>_T
\end{align}
Now applying  \eqref{proposition_abstract_2_2}, \eqref{proposition_abstract_2_3} to \eqref{proof_proposition_abstract_2_2} implies
\begin{align}
\label{proof_proposition_abstract_2_4}
|\<\<\Gamma_HF,G\>\>_T|\le s_2\norm{F}_{L^2_T\A}\norm{G}_{L^2_T\B}+s_3\norm{V_2\Gamma_HF}_{L^2_T\H}\norm{G}_{L^2_T\B}.
\end{align}
It remains to deal with $\norm{V_2\Gamma_HF}_{L^2_T\H}=\sup_{\norm{\widetilde G}_{L^2_T\H}=1}|\<\<V_2\Gamma_HF,\widetilde G\>\>_T|$. Since $D(H)$ is dense in $\H$, we may assume $\widetilde G(t)\in D(H)$. Then, taking $D(H)\subset D(V_2^*)$ into account, we use \eqref{proof_proposition_abstract_2_3} with $G=V_2^*\widetilde G$, \eqref{proposition_abstract_2_4}, \eqref{proposition_abstract_2_5} and \eqref{proposition_abstract_2_7} to obtain
$
|\<\<V_2\Gamma_HF,\widetilde G\>\>_T|\le (s_5+s_4s_7)\norm{F}_{L^2_T\A}
$
which, together with \eqref{proof_proposition_abstract_2_3}, gives us the second assertion. This completes the proof.
\end{proof}


\section{Proof of Theorems \ref{theorem_1} and \ref{theorem_2}}
\label{section_proof}
Let $H=-\Delta+V(x)$ be as in Theorem \ref{theorem_1}. This section is devoted to the proof of the main theorems. In what follows we use a standard notation
$2^*=2n/(n-2)$, 
$2_*=2n/(n+2)$.
 We write $\Gamma_0=\Gamma_{-\Delta}$. 
Let us first recall various estimates for the free Schr\"odinger equation. 

\begin{lemma}
\label{lemma_proof_1}
There exists $C>0$ such that, for any $v\in L^{n,\infty}(\R^n)$ and $T>0$, 
\begin{align}
\label{lemma_proof_1_1}
\norm{e^{it\Delta}\psi}_{L^2_TL^{2^*,2}}&\le C\norm{\psi}_{L^2},\\
\label{lemma_proof_1_2}
\norm{\Gamma_0 F}_{L^2_TL^{2^*,2}}&\le C \norm{F}_{L^2_TL^{2_*,2}},\\
\label{lemma_proof_1_3}
\norm{v(x)\Gamma_0 F}_{L^2_TL^{2}}&\le C \norm{v}_{L^{n,\infty}}\norm{F}_{L^2_TL^{2_*,2}},\\
\label{lemma_proof_1_4}
\norm{v(x)\Gamma_0 ^*F}_{L^2_TL^{2}}&\le C \norm{v}_{L^{n,\infty}}\norm{F}_{L^2_TL^{2_*,2}}. 
\end{align}
\end{lemma}

\begin{proof}
\eqref{lemma_proof_1_1} and \eqref{lemma_proof_1_2} are endpoint Strichartz estimates proved by \cite[Theorem 10.1]{KeTa}. The latter two estimates follow from \eqref{lemma_proof_1_2}, H\"older's inequality \eqref{Holder} and the duality.
\end{proof}

\begin{lemma}
\label{lemma_proof_2}
Let $w\in L^2(\R)$ be as in Theorem \ref{theorem_1}, $\rho>1/2$ and $v\in L^{n,\infty}(\R^n)$. Then there exists $C>0$, independent of $w,v$ and $T>0$, such that, for all $\psi\in L^2$ and simple function $F:\R\to \S(\R^n)$, one has
\begin{align}
\label{lemma_smoothing_2_1}
\norm{w(|x|)|D|^{1/2} e^{it\Delta}\psi}_{L^2_TL^2}
&\le C\norm{w}_{L^2(\R)}\norm{\psi}_{L^2},\\
\label{lemma_smoothing_2_2}
\norm{w(|x|)|D|^{1/2}\Gamma_0 F}_{L^2_TL^2}
&\le C \norm{w}_{L^2(\R)} \norm{F}_{L^2_TL^{2_*,2}},\\
\label{lemma_smoothing_2_3}
\norm{v(x)\Gamma_0 ^*F}_{L^2_TL^{2}}
&\le C \norm{v}_{L^{n,\infty}}\norm{w}_{L^2(\R)}\norm{w(|x|)^{-1}|D|^{-1/2}F}_{L^2_TL^{2}},\\
\label{lemma_smoothing_2_5}
\norm{\<x\>^{-\rho}|D|^{1/2} \Gamma_0 F}_{L^2_TL^2}
&\le C \norm{\<x\>^\rho |D|^{-1/2}F}_{L^2_TL^2},\\
\label{lemma_smoothing_2_6}
\norm{\<x\>^{-\rho}|D|^{1/2} \Gamma_0 F}_{L^2_TL^2}
&\le C \norm{F}_{L^2_TL^{2_*,2}},\\
\label{lemma_smoothing_2_7}
\norm{v(x)\Gamma_0 ^*F}_{L^2_TL^{2}}
&\le C \norm{v}_{L^{n,\infty}}\norm{\<x\>^{\rho}|D|^{-1/2}F}_{L^2_TL^{2}},\\
\label{lemma_smoothing_2_8}
\norm{v(x)\Gamma_0 F}_{L^2_TL^2}
&\le C \norm{v}_{L^{n,\infty}}\norm{\<x\>^\rho |D|^{-1/2}F}_{L^2_TL^{2}}.
\end{align}
\end{lemma}

\begin{proof}
Let us first consider \eqref{lemma_smoothing_2_1}. When $n=1$, it was proved by \cite{KPV}  that
$$
\sup_{x\in \R}\norm{|D_x|^{1/2}e^{it\partial_x^2}f}_{L^2(\R_t)}\le C \norm{f}_{L^2(\R_x)}
$$
which, together with the unitarity of $e^{it\Delta_{\widehat{x}_j}}$ in $L^2(\R^{n-1})$, implies
$$
\norm{|D_j|^{1/2} e^{it\Delta}\psi}_{L^\infty_{x_j}L^2_TL^2_{\widehat x_j}}
\le C\norm{\psi}_{L^2},\quad j=1,2,...,n,
$$
uniformly in $T>0$, where $\widehat x_j=(x_1,...,x_{j-1},x_{j+1},...,x_n)\in \R^{n-1}$ and $D_j=-i\partial_{x_j}$. \eqref{lemma_smoothing_2_1} then is derived from this estimate as follows. Let $\{C_j(\xi)\}$ be a conical partition of unity on $\R^n$ so that $
I=\sum_{j=1}^n C_j(\xi),
$
where $C_j\in C^\infty(\R^n\setminus\{0\})$ such that $\supp C_j\subset\{2|\xi_j|>|\xi|\}$ and $\partial_\xi^\alpha C_j(\xi)=O(|\xi|^{-|\alpha|})$. If we set 
$
\wtilde C_j(\xi)=C_j(\xi)|\xi|^{1/2}|\xi_j|^{-1/2}
$ then $\wtilde C_j$ also satisfies $\partial_\xi^\alpha C_j(\xi)=O(|\xi|^{-|\alpha|})$ and
$
|\xi|^{1/2}=\sum_{j=1}^n\wtilde C_j(\xi)|\xi_j|^{1/2}. 
$
Since $w(|x|)^2$ belongs to the Muckenhoupt $A_2$-class, $\wtilde C_j(D)$ is bounded on a weighted space $L^2(\R^n,w(|x|)^2dx)$ by weighted Mikhlin's multiplier theorem (see \cite{KuWh}). Thus we conclude that
\begin{align*}
\norm{w(|x|)|D|^{1/2}e^{it\Delta}\psi}_{L^2_TL^2}^2
&\le \sum_{j=1}^n\norm{w(|x|)\wtilde C_j(D)|D_{j}|^{1/2}e^{it\Delta}\psi}_{L^2_TL^2}^2\\
&\le C\sum_{j=1}^n\norm{w(|x|)|D_{j}|^{1/2}e^{it\Delta}\psi}_{L^2_TL^2}^2\\
&\le C\sum_{j=1}^n\norm{w}_{L^2(\R)}^2\norm{|D_{j}|^{1/2}e^{it\Delta}\psi}_{L^\infty_{x_j}L^2_TL^2_{\widehat x_j}}^2\\
&\le C \norm{w}_{L^2(\R)}^2\norm{\psi}_{L^2}^2
\end{align*}
uniformly in $T>0$, where we used the properties $w(|x|)\le C_jw(x_j)$ and $w\in L^2(\R)$ in the third line. Next, by the same argument as above, \eqref{lemma_smoothing_2_2} follows from the following estimate
\begin{align*}
\norm{|D_j|^{1/2}\Gamma_0 F}_{L^\infty_{x_j}L^2_TL^2_{\widehat x_j}}
\le C \norm{F}_{L^2_TL^{2_*,2}}.
\end{align*}
which is a slight generalization of \cite[Lemma 4]{IoKe}, in which the same estimate with $L^{2_*,2}$ replaced by $L^{2_*}$ was proved. Although the proof is essentially same as that of \cite[Lemma 4]{IoKe}, we briefly recall its strategy for reader's convenience. Without loss of generality, we may assume $j=1$. Then it suffices to show 
\begin{align}
\label{lemma_magnetic_1_2}
\sup_{x_1}\norm{|D_j|^{1/2}\wtilde\Gamma_0 F}_{L^2_TL^2_{\widehat x_1}}\le C\norm{F}_{L^2_TL^{2_*,2}},
\end{align}
where $\wtilde \Gamma_0 $ is defined by
$$
\wtilde \Gamma_0 F(t):=\int_{-\infty}^te^{i(t-s)\Delta}F(s)ds.
$$
Indeed, the corresponding estimate for $\Gamma_0 -\wtilde\Gamma_0 $ follows from \eqref{lemma_smoothing_2_1} and the dual estimate of \eqref{lemma_proof_1_1}. The only difference from the proof of \cite[Lemma 4]{IoKe} is an interpolation step. While they used the complex interpolation, we will use a real interpolation technique as in \cite[Section 6]{KeTa}. Let $I_{\pm}=\mathbf1_{\pm}(D_1)|D_1|^{1/2}\wtilde\Gamma_0 $, where $\mathbf 1_\pm(t)=1$ for $\pm t\ge0$ and $\mathbf 1_\pm(t)=0$ for $\mp t\ge0$. 
It suffices to show that $I_{+}$ is bounded from $L^2_TL^{2_*,2}$ to $L^2_TL^2_{\widehat x_1}$ uniformly in $x_1$ since the proof for $I_{-}$ being analogous. By the $TT^*$ argument, $I_{+}\in \mathbb B(L^2_TL^{2_*,2},L^2_TL^2_{\widehat x_1})$ if $I_{+}^*I_{+}$ is bounded from $L^2_TL^{2_*,2}$ to $L^2_TL^{2^*,2}$. Hence, if we define a bilinear form $I$ by$$
I(F,G):=\iint \<I_{+}F(s,\cdot),I_{+}G(t,\cdot)\>dsdt
$$
then it suffices to show that 
\begin{align}
\label{lemma_magnetic_1_3}
|I(F,G)|\le C\norm{F}_{L^2_TL^{2_*,2}}\norm{G}_{L^2_TL^{2_*,2}}
\end{align}
uniformly in $x_1$ and $T>0$. It was shown by \cite{IoKe} that $I_{+}^*I_{+}$ is bounded on $L^2(\R^n)$ and the kernel of $I_{+}^*I_{+}$, denoted by $K_{+}(t,s,x,y)$, satisfies the dispersive estimate:
$$
|K_{+}(t,s,x,y)|\le C|t-s|^{-n/2},\quad t\neq s.
$$
We then decompose $I(F,G)$ as
$$
I(F,G)=\sum_{k\in\Z}I_k(F,G),\ I_k(F,G):=\iint_{t-2^{k+1}}^{t-2^k} \<I_{+}F(s,\cdot),I_{+}G(t,\cdot)\>dsdt. 
$$
By using the same argument as in \cite[Lemma 4.1]{KeTa}, we see that 
\begin{align*}
|I_k(F,G)|\le  C2^{-k\beta(a,b)}\norm{F}_{L^2_TL^{a'}}\norm{G}_{L^2_TL^{b'}},\quad \beta(a,b)=\frac n2-1-\frac n2(\frac1a-\frac1b)
\end{align*}
uniformly in $k\in \Z$, where $(a,b)$ satisfies one of the following conditions:
$$
\text{(i)\ \ $\frac 1a=\frac 1b=0$;\quad (ii)\ \ $\frac{n-1}{2n}\le \frac1a \le \frac12$ and $\frac 1b=\frac12$;\quad (iii)\ \ $\frac{n-1}{2n}\le \frac1b \le \frac12$ and $\frac 1a=\frac12$. }
$$
In other words, a vector valued sequence $(I_k)_{k\in \Z}$ is bounded from $L^2_TL^{a'}_x\times L^2_TL^{b'}_x$ to $\ell_{\beta(a,b)}^\infty$, where $\ell^p_s=L^p(\Z, 2^{js}dj)$ is a weighted $\ell^p$ space with the counting measure $dj$. Then \eqref{lemma_magnetic_1_3} follows from the technique by \cite[Section 6]{KeTa} based on a bilinear real interpolation. 

The estimate \eqref{lemma_smoothing_2_3} follows from the dual estimate of \eqref{lemma_smoothing_2_2} and H\"older's inequality  \eqref{Holder}. For \eqref{lemma_smoothing_2_5}, we refer to \cite{Chi}. \eqref{lemma_smoothing_2_6} and \eqref{lemma_smoothing_2_7} follow from \eqref{lemma_smoothing_2_2} and \eqref{lemma_smoothing_2_3} since $\<x\>^{-\rho}$ satisfies the condition on $w$ in Theorem \ref{theorem_1}. In order to derive \eqref{lemma_smoothing_2_8}, we observe from the formula \eqref{adjoint_1} that $\Gamma_0 $ can be brought to the form
\begin{align*}
\Gamma_0 F(t)=-\Gamma_0 ^*F(t)\pm\Gamma_0 ^{0}F(t)+\Gamma_0 ^{\mp}F(t)
\end{align*}
for $\pm t\ge0$, where 
$$
\Gamma_0 ^{0}F(t)=\int_{-T}^Te^{-i(t-s)H}F(s)ds,\quad
\Gamma_0 ^{\pm}F(t)=\int_0^{\pm T}e^{-i(t-s)H}F(s)ds. 
$$
Then the desired estimate for $\Gamma_0 ^*$ is nothing but \eqref{lemma_smoothing_2_6}; the desired estimates for $\Gamma_0 ^{0}$ and $\Gamma_0 ^{\mp}$ follow from \eqref{lemma_proof_1_1}, the dual estimate of \eqref{lemma_smoothing_2_1} with $w=\<x\>^{-\rho}$ and H\"older's inequality  \eqref{Holder}. 
\end{proof}

The following fact, proved by \cite[Theorem 1.6 and (1.23)]{BVZ} (see also \cite[Theorem 6.1 and Appendix B]{BoMi}), also  plays an important role.
\begin{proposition}
\label{proposition_smoothing_3}
Let $n\ge3$ and $V$ satisfy Assumption \ref{assumption_A}. Then $|x|^{-1}$ is $H$-supersmooth.
\end{proposition}

We are in a position to show the main theorems. 

\begin{proof}[Proof of Theorem \ref{theorem_1}]
Let us set $V_1=|x|V$ and $V_2=|x|^{-1}$. By Sobolev's inequality \eqref{Sobolev}, 
$$
\norm{V_jf}_{L^2}\le C \norm{V_j}_{L^{n,\infty}}\norm{f}_{L^{2^*,2}}\le C \norm{V_j}_{L^{n,\infty}}\norm{\nabla f}_{L^2}
$$
and hence $D(V_j)\supset \H^1\supset D(\Delta)\cup D(H)$. Moreover,  \eqref{lemma_proof_1_1} and Proposition \ref{proposition_abstract_1} (1) show that both $V_1$ and $V_2$ are $\Delta$-smooth. On the other hand, Propositions \ref{proposition_abstract_1} and  \ref{proposition_smoothing_3} show
\begin{align}
\label{proof_1_1}
\norm{V_2e^{-itH}\psi}_{L^2_TL^2}\le C \norm{\psi}_{L^2}
\end{align}
uniformly in $T>0$. Let $\B$ the completion of $C_0^\infty$ with respect to the norm $\norm{w(|x|)^{-1}|D|^{-1/2}f}_{L^2}$. By virtue of \eqref{lemma_smoothing_2_1}, \eqref{lemma_smoothing_2_3} and \eqref{proof_1_1}, one can use Proposition \ref{proposition_abstract_2} with $H_0=-\Delta$, $H=-\Delta+V$ and this $\B$ to obtain
$$
|\<\<e^{-itH}\psi,G\>\>_T|\le C \norm{w}_{L^2(\R)}\norm{\psi}_{L^2}\norm{G}_{\B}
$$
for all $\psi\in L^2$ and simple function $G:[-T,T]\to \S$ uniformly in $T>0$. Then the desired estimate follows from density and duality arguments. 
\end{proof}

\begin{proof}[Proof of Theorem \ref{theorem_2}]
We use the same decomposition $V=V_1V_2$ as above. Since $V_2$ is $H$-supersmooth, we learn by  Proposition \ref{proposition_abstract_1} and a remark after Lemma \ref{lemma_abstract_2} that 
\begin{align}
\label{proof_2_1}
\norm{V_2\Gamma_HV_2\wtilde G}_{L^2_TL^2}\le C\norm{\wtilde G}_{L^2_TL^2}
\end{align}
for all simple function $\wtilde G:\R\to D(V_2)$ with the constant $C$ independent of $T$ and $\wtilde G$. By virtue of \eqref{lemma_proof_1_1}--\eqref{lemma_proof_1_4}, \eqref{lemma_smoothing_2_1} with $w=\<x\>^{-\rho}$, \eqref{lemma_smoothing_2_5}--\eqref{lemma_smoothing_2_8} with $v\in \{V_1,V_2\}$, \eqref{proof_1_1} and \eqref{proof_2_1}, we can use Proposition \ref{proposition_abstract_2} with $\A,\B\in \{\dot\H^{-1/2,\rho},L^{2_*,2}\}$ to obtain 
\begin{align*}
|\<\<e^{-itH}\psi,G\>\>_T|\le C \norm{\psi}_{L^2}\norm{G}_{L^2_T\B},\quad
|\<\<\Gamma_HF,G\>\>_T|\le C \norm{F}_{L^2_T\A}\norm{G}_{L^2_T\B}
\end{align*}
uniformly in $T>0$, $\psi\in L^2$ and simple functions $F,G:\R\to\S$. Then the assertion follows from density of simple functions $F:\R\to\S$ in $L^2_T\A$ and $L^2_T\B$ and the formula \eqref{solution}. 
\end{proof}

\end{document}